\documentclass[11pt,reqno]{amsart}

\usepackage[T1]{fontenc}
\usepackage{lmodern}
\usepackage{amsmath,amssymb,mathtools}
\usepackage{amsthm}
\usepackage{microtype}
\usepackage{csquotes}
\usepackage[style=numeric,sorting=nyt,giveninits=true,maxbibnames=99]{biblatex}
\usepackage[hidelinks]{hyperref}

\numberwithin{equation}{section}

\theoremstyle{plain}
\newtheorem{theorem}{Theorem}[section]
\newtheorem{proposition}[theorem]{Proposition}
\newtheorem{lemma}[theorem]{Lemma}
\newtheorem{corollary}[theorem]{Corollary}
\newtheorem{conjecture}[theorem]{Conjecture}
\theoremstyle{definition}

\theoremstyle{remark}
\newtheorem{remark}[theorem]{Remark}

\newtheorem{openproblem}[theorem]{Open Problem}

\newcommand{\N}{\mathbb N}
\newcommand{\Z}{\mathbb Z}

\title[DIVISIBILITY OF SHIFTED SUM-OF-DIVISORS VALUES]{On the Divisibility Relation $\sigma(n)\mid\sigma(n+h)$ and a Generalized Erd\H{o}s--Sierpi\'nski Conjecture}
\author{Amirali Fatehizadeh}
\address{Faculty of Mathematical Sciences, Shahid Beheshti University, Tehran, Iran}
\email{Amirali.fatehizadeh@gmail.com}
\email{A.fatehizadeh@mail.sbu.ac.ir}

\author{Florian Luca}
\address{Mathematics Division, Stellenbosch University, Stellenbosch, South Africa}
\address{Department of Computer Science, University of Oxford, Oxford, United Kingdom}
\email{fluca@sun.ac.za}
\date{}


\begin{document}
\begin{abstract}
For each fixed positive integer $h$, we study the divisibility relation $\sigma(n)\mid\sigma(n+h)$. We isolate an explicit regular family arising from integral quotients of shifted abundancy indices and show that the complementary set satisfies a subexponential saving; in particular, the number of solutions up to $x$ is $O_h(x/(\log x)^2)$. We also study the proportionality equation $\sigma(n+h)=\lambda\sigma(n)$. For every fixed nonzero integer $h$, uniformly for all real $\lambda>0$, the number of solutions up to $x$ is $O(x/\sqrt{\log\log\log x})$, with an absolute implied constant once $x$ exceeds an $h$-dependent threshold. Finally, we give an explicit family which, under Schinzel's Hypothesis $H$, produces infinitely many solutions of $\sigma(n+1)=2\sigma(n)$; the Bateman--Horn conjecture yields a precise asymptotic for the number of members of this family up to $x$. We conjecture that $\sigma(n+h)=k\sigma(n)$ has infinitely many positive integer solutions for every fixed $h,k\ge1$.
\end{abstract}

\maketitle
\enlargethispage{2pt}

\section{Introduction}\label{sec:introduction}

Let $\sigma(n)=\sum_{d\mid n}d$ denote the sum-of-divisors function.
If $n=\prod_p p^{a_p}$, then
$\sigma(n)=\prod_p(1+p+\cdots+p^{a_p})$, so the value of $\sigma(n)$
is determined directly by the prime-power factorization of $n$.
This multiplicative description gives little direct control, however,
when one compares values at additively related arguments. For a fixed
nonzero shift $h$, the factorizations of $n$ and $n+h$ are related only
through the additive constraint between the two integers. Consequently,
relations between $\sigma(n)$ and $\sigma(n+h)$ combine the multiplicative
structure of $\sigma$ with an additive constraint. Equal values, fixed
ratios, and divisibility between these shifted values are therefore natural
problems in the joint value distribution of $\sigma$.

A classical starting point is the problem of consecutive equal values. Sierpi\'nski asked whether
\[
\sigma(n+1)=\sigma(n)
\]
holds for infinitely many positive integers $n$; Guy and Shanks also record an earlier conjecture of Erd\H{o}s asserting the same infinitude \cite{GuyShanks1974}. Numerical evidence preceded any general analytic estimate. M\k{a}kowski listed the nine solutions below $10^4$; his original paper was followed by a correction in 1961 \cite{Makowski1960,MakowskiCorr1961}. Hunsucker, Nebb and Stearns subsequently extended the computation to $10^7$, obtaining $113$ solutions in that range \cite{HunsuckerNebbStearns1973}. Guy and Shanks then gave an explicit large solution by exploiting a prime-factor pattern visible among smaller examples. Their construction recovered several earlier solutions from the same pattern, although they observed that it did not appear likely, by itself, to yield infinitely many solutions \cite{GuyShanks1974}. The corresponding sequence of solutions is OEIS A002961
\cite{OEIS2026}. Thus the early history already exhibited two complementary features of the problem: substantial computational evidence and rigid multiplicative patterns among individual solutions.
Later computations extended the numerical record substantially. Benito found $1268$ solutions for $n\le 1.5\cdot 10^{10}$ \cite{Benito2007}, while Bayless and Kinlaw report the Noe--Resta computation of $10135$ solutions through $10^{13}$ and also obtained explicit upper and lower bounds for the reciprocal sum of the consecutive solution set \cite{BaylessKinlaw2015}.

The consecutive question was also placed in a fixed-shift setting. Mientka and Vogt asked for which integers $h\ge1$ the equation
\[
\sigma(n+h)=\sigma(n)
\]
has infinitely many solutions \cite{MientkaVogt1970}. Hunsucker, Nebb and Stearns subsequently found at least two solutions for every shift $h\le5000$ \cite{HunsuckerNebbStearns1973,Weingartner2011}. The consecutive problem is therefore part of a wider family of additive-shift questions in which the existence and arithmetic form of solutions may depend on the shift. This fixed-shift perspective is fundamental below: the divisibility results of Section~\ref{sec:divisibility} are formulated for every fixed positive $h$, whereas the uniform proportionality theorem of
Section~\ref{sec:levelsets} applies to every fixed nonzero integer shift.

A different direction emerged from quantitative analytic estimates. Erd\H{o}s, Pomerance and S\'ark\H{o}zy proved that
\[
\#\{n\le x:\sigma(n+1)=\sigma(n)\}
\ll
x\exp\!\bigl(-(\log x)^{1/3}\bigr).
\]
They nevertheless conjectured that, for every $\varepsilon>0$, the number of solutions up to $x$ should eventually exceed $x^{1-\varepsilon}$ \cite{ErdosPomeranceSarkozy1987}. Thus a strong sparsity estimate remains compatible with the expected infinitude of the solution set. Their work shifted the emphasis from computation and isolated constructions toward analytic control of the entire set of solutions, a distinction that remains relevant in later fixed-shift problems.

Further work revealed a substantial structured component in shifted equalities. Yamada studied equations
\[
\sigma(a_1n+b_1)=\sigma(a_2n+b_2)
\]
for affine forms, constructed families arising from fixed arithmetic cores and simultaneous primality conditions, and obtained a subexponential upper bound for the solutions not generated by those families \cite{Yamada2017}. In the parallel shifted-equality case, equal-abundancy relations are central to these structured families. Weingartner, approaching the existence problem from a different direction, proved under Schinzel's Hypothesis $H$ that $\sigma(n+h)=\sigma(n)$ has infinitely many solutions for every fixed positive even $h$ \cite{Weingartner2011}. Equal-abundancy relations also play a central role in Ford's later construction. Writing $I(n)=\frac{\sigma(n)}{n}$ for the abundancy index, integers with the same value of $I$ are traditionally called friends, and Pollack and Pomerance developed quantitative and structural aspects of friendly sets \cite{PollackPomerance2016}. Ford combined large equal-abundancy collections with modern results on prime tuples to prove unconditionally that $\sigma(n+h)=\sigma(n)$ has infinitely many solutions for a positive proportion of integers $h\ge1$ \cite{Ford2022}.

The equality condition is one level of a broader problem. The equation
\[
\sigma(n+1)=k\sigma(n)
\]
was studied by Torabi and Fatehizadeh by elementary and existential methods \cite{TorabiFatehizadeh2021}. For a fixed positive shift $h$, equality is precisely the level $k=1$ of the integer-multiplier problem, while every integer-multiplier solution is automatically a divisibility solution. More precisely,
\[
\{n\in\N:\sigma(n+h)=\sigma(n)\}
\subseteq
\bigcup_{k\in\N}\{n\in\N:\sigma(n+h)=k\sigma(n)\}
=
\mathcal B_h,
\]
where
\[
\mathcal B_h=\{n\in\N:\sigma(n)\mid\sigma(n+h)\},
\qquad
B_h(x)=\#(\mathcal B_h\cap[1,x]).
\]
For the consecutive shift $h=1$, the elements of $\mathcal B_1$
form OEIS sequence A058072 \cite{OEIS2026}. The structural set that arises in the analysis of $\mathcal B_h$ is
\[
\mathcal V_h
=
\left\{v\ge2:\frac{I(v)}{I(v+h)}\in\N\right\}.
\]
Its ratio-one slice $I(v)=I(v+h)$ is the equal-abundancy condition
occurring in the structured shifted-equality constructions just described,
whereas $\mathcal V_h$ allows an arbitrary positive integral quotient of
two shifted abundancy indices. For later comparison with arbitrary proportionality levels, if $h\in\Z\setminus\{0\}$ and $\lambda>0$, we also write
\[
A_{h,\lambda}(x)=\#\{n\le x:n+h\ge1,\ \sigma(n+h)=\lambda\sigma(n)\}.
\]

Our principal result concerns the full divisibility set $\mathcal B_h$. Its simplest consequence is
\[
B_h(x)\ll_h\frac{x}{(\log x)^2}
\]
for every fixed $h\ge1$. The stronger statement isolates explicitly the part of $\mathcal B_h$ arising from the integral-abundancy resonances in $\mathcal V_h$. For $v\in\mathcal V_h$, writing $\delta=(v,h)$, the corresponding regular solutions arise from simultaneous primality of
\[
\frac{v}{\delta}L-1
\qquad\text{and}\qquad
\frac{v+h}{\delta}L-1,
\]
together with the additional arithmetic conditions specified in Section~\ref{sec:divisibility}. Let $\mathcal R_h$ denote this regular family, and put $R_h(x)=\#(\mathcal R_h\cap[1,x])$.

Theorem~\ref{thm:regular-decomposition} states that, for every fixed integer
$h\ge1$,
\[
\mathcal R_h\subseteq\mathcal B_h,
\qquad
R_h(x)\ll_h\frac{x}{(\log x)^2},
\]
and
\[
0\le B_h(x)-R_h(x)
\le
x\exp\!\left(
-\left(\frac1{\sqrt6}+o_h(1)\right)
\sqrt{\log x\,\log\log x}
\right)
\]
as $x\to\infty$. Thus the theorem separates an explicit regular source of divisibility
solutions from a complementary set satisfying a subexponential upper bound. It also yields a dichotomy governed by $\mathcal V_h$. If $\mathcal V_h=\varnothing$, the subexponential term above bounds all of $B_h(x)$; if $\mathcal V_h\ne\varnothing$, the Bateman--Horn conjecture applied to an associated pair of affine-linear forms gives
\[
B_h(x)\asymp_h\frac{x}{(\log x)^2}.
\]
Since every integer-multiplier level is contained in $\mathcal B_h$, Theorem~\ref{thm:regular-decomposition} also gives, uniformly over $k\in\N$,
\[
A_{h,k}(x)\ll_h\frac{x}{(\log x)^2}.
\]

There is a complementary line of work in which the emphasis is on joint
values of multiplicative functions in substantially greater generality.
Mangerel's Theorem~1.1 treats fixed scalar relations between multiplicative
functions evaluated on two primitive nonproportional affine forms and
proves, under a divergence condition on prime values, that the corresponding
solution set has logarithmic density zero. Remark~1.2 explains the
restriction to logarithmic density: the argument relies on Tao's
logarithmically averaged binary-correlation theorem, while no unconditional
Ces\`aro-averaged result of the required strength is available. The proof
passes through the twists $|f|^{it}$ and combines logarithmically averaged
non-pretentious correlation estimates with inverse-sumset input from
continuous additive combinatorics.

For $f_1=f_2=\sigma$, the divergence hypothesis is immediate from
$\sigma(p)=p+1$. For a positive fixed shift $h$, Theorem~1.1 therefore
applies directly to the relation $\sigma(n+h)=\lambda\sigma(n)$ for any
fixed $\lambda>0$, and yields logarithmic-density zero. If $h<0$, write
$H=-h>0$ and set $m=n-H$. Then
\[
\sigma(n+h)=\lambda\sigma(n)
\]
is equivalent to
\[
\sigma(m+H)=\lambda^{-1}\sigma(m).
\]
Thus the negative-shift case follows from the positive-shift case by
reindexing and replacing the level by its reciprocal. Mangerel's theorem
therefore gives logarithmic-density zero for every fixed nonzero shift and
every fixed proportionality level. It does not, however, provide a
quantitative natural-density estimate uniform in the proportionality
parameter \cite{Mangerel2024}.

The latter is provided by Theorem~\ref{thm:uniform-levelset}. It states that there exists an absolute constant $C>0$ such that, for every nonzero integer $h$, there exists $x_0(h)$ for which, for every $x\ge x_0(h)$,
\[
\sup_{\lambda>0}A_{h,\lambda}(x)
\le
C\frac{x}{\sqrt{\log\log\log x}}.
\]
The constant $C$ is absolute; only the threshold $x_0(h)$ depends on the shift, and the estimate is simultaneous in all real $\lambda>0$.

For positive $h$ and integer $k$, the $x/(\log x)^2$ consequence of
Theorem~\ref{thm:regular-decomposition} is quantitatively stronger. The scope of Theorem~\ref{thm:uniform-levelset} is different: it applies to every fixed nonzero shift, including negative shifts, and to every real level $\lambda>0$, uniformly in $\lambda$. It therefore supplies a quantitative natural-density estimate in a setting in which Mangerel's much broader multiplicative-function theorem gives a qualitative logarithmic-density conclusion. The two results address different aspects of the same level-set problem.

The proofs of Theorems~\ref{thm:regular-decomposition} and~\ref{thm:uniform-levelset} also have distinct structures. The argument in Section~\ref{sec:divisibility} begins by isolating large prime factors of $n$ and $n+h$. After removing the smooth cases and the large-prime-square exceptional
set, the divisibility condition leads to a normalized Diophantine equation
relating the two large primes. Nonresonant configurations are estimated by divisor bounds and sieve arguments, while the unique resonant configuration forces an integral quotient of shifted abundancy indices and produces two affine-linear prime forms. This is the origin of $\mathcal V_h$ and of the regular family $\mathcal R_h$: the resonant configuration is precisely the structure that remains after the nonresonant configurations have been estimated. A two-linear-form upper-bound sieve then controls the resulting regular families.

Section~\ref{sec:levelsets} proceeds independently. Writing
\[
g(n)=\log\frac{\sigma(n)}{n},
\]
a relation between $\sigma(n+h)$ and $\sigma(n)$ becomes a level-set condition for $g(n+h)-g(n)$. After truncating $g$ to bounded prime powers, the truncated difference is exactly periodic. Uniform measure on a complete period admits an exact product decomposition, furnished by the Chinese remainder theorem, into independent local prime coordinates. The resulting local variables have sufficient dispersion at primes not dividing $h$ for the Kolmogorov--Rogozin concentration inequality to apply. Primes dividing the fixed shift remain in the exact product model but are placed in an independent convolution remainder; the anti-concentration
estimate is obtained from the remaining prime coordinates. This method does not use the integrality, or even the rationality, of $\sigma(n+h)/\sigma(n)$; accordingly, it applies to all real $\lambda>0$ and to negative as well as positive fixed shifts.

The quantitative sparsity results are complemented in Section~\ref{sec:conditional} by a conditional existence theorem for multiplier $2$. Consider
\[
n_t=(252t+223)(6t+5),
\qquad
n_t+1=6(7t+6)(36t+31).
\]
Whenever the four displayed affine-linear factors are simultaneously prime,
\[
\sigma(n_t+1)=2\sigma(n_t).
\]
Schinzel's Hypothesis $H$ therefore gives infinitely many distinct solutions from this family. Under the Bateman--Horn conjecture, the counting function $F(x)$ defined in
Section~\ref{sec:conditional} satisfies
\[
F(x)
\sim
\frac{8\mathfrak S}{3\sqrt{42}}
\frac{\sqrt{x}}{(\log x)^4},
\qquad
\mathfrak S>0.
\]
The same family has additional divisor-partition properties that follow directly from its particular factorization, but no analogous structural conclusion is asserted for arbitrary solutions of
\[
\sigma(n+h)=k\sigma(n).
\]
The coexistence of quantitative sparsity with conditional infinitude motivates the fixed-shift extension of the Erd\H{o}s--Sierpi\'nski conjecture formulated at the end of the paper:
\[
\#\{n\ge1:\sigma(n+h)=k\sigma(n)\}=\infty
\qquad(h,k\ge1).
\]
The case $(h,k)=(1,1)$ is the classical Erd\H{o}s--Sierpi\'nski conjecture. We also formulate the more specific conjecture $\mathcal V_1=\varnothing$, concerning the integral-abundancy resonance underlying the structural
dichotomy in Section~\ref{sec:divisibility}.

Section~\ref{sec:preliminaries} collects the analytic, probabilistic, and
divisor-partition preliminaries. Section~\ref{sec:divisibility} proves the fixed-shift divisibility theorem and develops the associated integral-abundancy structure. Section~\ref{sec:levelsets} establishes the uniform proportionality estimate
by means of the exact periodic product model and an anti-concentration
argument. Section~\ref{sec:conditional} gives the conditional multiplier-two construction, its restricted structural consequences, and the conjectures arising from the preceding results.

\section{Preliminaries}\label{sec:preliminaries}

We collect the notation and the auxiliary analytic, probabilistic, and divisor-partition results used throughout the paper. We write $\mathbb N=\{1,2,\ldots\}$, and all logarithms are natural.

\subsection{Analytic preliminaries}

For $n\in\mathbb N$, let
\[
\sigma(n)=\sum_{d\mid n}d,
\qquad
\tau(n)=\sum_{d\mid n}1.
\]

For $n>1$, let $P^+(n)$ denote the largest prime factor of $n$, with $P^+(1)=1$. We write
\[
I(n)=\frac{\sigma(n)}{n}
\]
for the abundancy index. As usual, $f\ll g$ means $f=O(g)$; subscripts indicate the parameters on which the implied constant may depend.

For $X\ge3$ and $2\le y\le X$, define
\[
\Psi(X,y)=\#\{n\le X:P^+(n)\le y\},
\qquad
u=\frac{\log X}{\log y}.
\]
We shall use the estimate
\begin{equation}
\Psi(X,y)
\le
X\exp\{-(1+o(1))u\log u\}.
\label{eq:smooth}
\end{equation}
The parameters used in Section~\ref{sec:divisibility} satisfy
\[
\log y\asymp\sqrt{\log X\,\log\log X},
\]
and hence $u\to\infty$. Estimate~\eqref{eq:smooth} is valid in this range by standard smooth-number estimates; see Canfield, Erd\H{o}s and Pomerance
\cite{CanfieldErdosPomerance1983} and Hildebrand and Tenenbaum
\cite[Corollary~1.3]{HildebrandTenenbaum1993}.

We shall also use
\begin{equation}
I(n)
\le
\prod_{p\mid n}
\left(1-\frac1p\right)^{-1}
\ll
\log\log(3n).
\label{eq:abundancy-bound}
\end{equation}
We use the classical estimates
\[
\prod_{p\le y}\left(1-\frac1p\right)\asymp\frac1{\log y},
\qquad
\sum_{p\le y}\frac1p=\log\log y+O(1),
\]
together with
\[
\pi(y)\ll\frac{y}{\log y},
\qquad
\vartheta(y):=\sum_{p\le y}\log p\le2y
\]
for all sufficiently large $y$. Partial summation then gives
\[
\sum_{p>y}\frac1{p^2}\ll\frac1{y\log y}.
\]
We shall also use the elementary estimate
\[
\sum_{n\le X}\tau(n)\ll X\log(2X).
\]

We also need an average bound for $\tau(\sigma(n))$. Luca and Pomerance note that the corresponding estimate for $\tau(\varphi(n))$ remains valid with $\varphi$ replaced by $\sigma$; we use this result together with their corrigendum
\cite{LucaPomerance2007,LucaPomeranceCorr2016}.
Thus there is an absolute constant $C_0>0$ such that, for all sufficiently large $X$,
\begin{equation}
\sum_{n\le X}\tau(\sigma(n))
\le
X\exp\left(
C_0\sqrt{\frac{\log X}{\log\log X}}
\right).
\label{eq:luca-pomerance}
\end{equation}

Partial summation gives
\begin{equation}
\sum_{n\le X}\frac{\tau(\sigma(n))}{n}
\ll
\log X\,
\exp\left(
C_0\sqrt{\frac{\log X}{\log\log X}}
\right).
\label{eq:partial-sum-tau-sigma}
\end{equation}

We also use the Selberg upper-bound sieve for two affine-linear forms. The precise specialization required below, including the local factors, is stated and proved in Lemma~\ref{lem:two-linear-sieve}; for the underlying sieve, see
\cite[Chapter~2]{HalberstamRichert1974}.

\subsection{Probabilistic preliminaries}

For a real-valued random variable $X$ and $L>0$, define its concentration function by
\[
Q_L(X)
=
\sup_{a\in\mathbb R}
\Pr\{X\in[a,a+L]\}.
\]
The following elementary monotonicity property will be useful.

\begin{lemma}\label{lem:convolution}
If $X$ and $Y$ are independent real-valued random variables, then
\[
Q_L(X+Y)\le Q_L(X).
\]
\end{lemma}

\begin{proof}
For every $a\in\mathbb R$,
\[
\begin{aligned}
\Pr\{X+Y\in[a,a+L]\}
&=
\mathbb E\!\left[
\Pr\{X\in[a-Y,a+L-Y]\mid Y\}
\right]\\
&\le Q_L(X).
\end{aligned}
\]
Taking the supremum over $a$ proves the result.
\end{proof}

We use the following form of the Kolmogorov--Rogozin inequality.

\begin{theorem}[Kolmogorov--Rogozin]\label{thm:KR}
Let $X_1,\ldots,X_s$ be independent real-valued random variables, and put $S=X_1+\cdots+X_s$. If $L>0$ and
\[
0<L_j\le2L
\qquad (1\le j\le s),
\]
then
\[
Q_L(S)
\le
C_{\mathrm{KR}}L
\left(
\sum_{j=1}^s
L_j^2\bigl(1-Q_{L_j}(X_j)\bigr)
\right)^{-1/2},
\]
where $C_{\mathrm{KR}}>0$ is an absolute constant.
\end{theorem}

We shall use only the equal-scale case $L_j=L$, which gives
\[
Q_L(S)
\le
C_{\mathrm{KR}}
\left(
\sum_{j=1}^s
\bigl(1-Q_L(X_j)\bigr)
\right)^{-1/2}.
\]
This is Rogozin's inequality in the form reproduced by Kesten
\cite[Theorem~1, equation~(1.2)]{Kesten1969}.

In Section~\ref{sec:levelsets}, the required independence is exact: on a complete residue period, the Chinese remainder theorem identifies the uniform measure with a product of local uniform measures.

\subsection{Divisor-partition preliminaries}

Let $D(n)$ denote the set of positive divisors of $n$.

A positive integer $n$ is \emph{perfect} if $\sigma(n)=2n$, and is \emph{$k$-multiperfect}, for an integer $k\ge2$, if $\sigma(n)=kn$. It is \emph{abundant} if $\sigma(n)>2n$.

A positive integer $n$ is \emph{Zumkeller} if $D(n)$ can be partitioned into two subsets having equal sum. More generally, for $k\ge2$, the integer $n$ is \emph{$k$-layered} if there is a partition
\[
D(n)=D_1\sqcup\cdots\sqcup D_k
\]
such that
\[
\sum_{d\in D_1}d
=
\cdots
=
\sum_{d\in D_k}d
=
\frac{\sigma(n)}{k}.
\]
Thus the Zumkeller numbers are precisely the $2$-layered numbers. If $n$ is $k$-layered, then necessarily $k\mid\sigma(n)$ and $\sigma(n)\ge kn$; the latter follows from the class containing the divisor $n$.

Every perfect number is Zumkeller, since $\{n\}$ and the set of proper divisors of $n$ both have sum $n$.

Mahanta, Saikia and Yaqubi study Zumkeller and $k$-layered numbers, including their relations with practical numbers
\cite{MahantaSaikiaYaqubi2020}. Further results on $k$-layered numbers and their relation to $k$-multiperfect numbers are given by Jokar
\cite{Jokar2022}.

A positive integer $n$ is \emph{practical} if every integer $m$ with $1\le m\le n$ can be represented as a sum of distinct positive divisors of $n$. We shall use the following classical characterization.

\begin{theorem}[Stewart--Sierpi\'nski]\label{thm:practical}
Let
\[
n=p_1^{\alpha_1}\cdots p_r^{\alpha_r},
\qquad
p_1<\cdots<p_r,
\]
with $n\ge2$. Then $n$ is practical if and only if $p_1=2$ and
\[
p_j
\le
1+
\sigma\!\left(
\prod_{i<j}p_i^{\alpha_i}
\right)
\qquad (2\le j\le r).
\]
\end{theorem}

This characterization is due independently to Stewart and Sierpi\'nski
\cite{Stewart1954,Sierpinski1955}.

Finally, $n$ is \emph{semiperfect} if it can be represented as a sum of distinct proper divisors. A \emph{weird number} is an abundant integer that is not semiperfect. Thus every abundant semiperfect number is non-weird.

No general divisor-partition conclusion is drawn from an equation
\[
\sigma(n+h)=k\sigma(n).
\]
The structural properties used in Section~\ref{sec:conditional} are proved directly for the explicit family considered there.

\section{Fixed-shift divisibility and integral abundancy structure}\label{sec:divisibility}

Throughout this section, $h\ge1$ is fixed. Define
\[
\mathcal B_h=\{n\in\N:\sigma(n)\mid\sigma(n+h)\},
\qquad
B_h(x)=\#(\mathcal B_h\cap[1,x]).
\]

\subsection{Structural setup and auxiliary results}

For $v\ge2$, put
\[
\delta_v=(v,h),
\qquad
a_v=\frac{v}{\delta_v},
\qquad
b_v=\frac{v+h}{\delta_v}.
\]
Then $(a_v,b_v)=1$ and $b_v-a_v=h/\delta_v$.

Define the exceptional set by
\begin{equation}\label{eq:Vh}
\mathcal V_h
=
\left\{v\ge2:\frac{I(v)}{I(v+h)}\in\N\right\}.
\end{equation}
Since
\[
\frac{I(v)}{I(v+h)}
=
\frac{(v+h)\sigma(v)}{v\sigma(v+h)},
\]
condition \eqref{eq:Vh} is equivalent to
\begin{equation}\label{eq:Vh-divisibility}
v\sigma(v+h)\mid(v+h)\sigma(v).
\end{equation}

For $v\in\mathcal V_h$ and $L\ge1$, set
\[
p_{v,L}=a_vL-1,
\qquad
q_{v,L}=b_vL-1.
\]
Let $\mathcal R_h$ be the set of integers
\begin{equation}\label{eq:regular-n}
n=(v+h)p_{v,L},
\end{equation}
where $v\in\mathcal V_h$, both $p_{v,L}$ and $q_{v,L}$ are prime, and
\begin{equation}\label{eq:cross-coprime}
(p_{v,L},v+h)=1,
\qquad
(q_{v,L},v)=1.
\end{equation}
Write $R_h(x)=\#(\mathcal R_h\cap[1,x])$.

\begin{proposition}\label{prop:integral-abundancy}
If $v\in\mathcal V_h$, then
\[
\frac{v}{(v,h)}\mid\sigma(v).
\]
In particular, if $(v,h)=1$, then $v$ is multiperfect.
\end{proposition}
\begin{proof}
Put $\delta=(v,h)$, $a=v/\delta$, and $b=(v+h)/\delta$. By \eqref{eq:Vh}, there is an integer $r\ge1$ such that
\[
\frac{I(v)}{I(v+h)}=r.
\]
Equivalently,
\[
b\sigma(v)=ra\sigma(v+h).
\]
Hence $a\mid b\sigma(v)$. Since $(a,b)=1$, it follows that $a\mid\sigma(v)$, proving the first assertion.

If $(v,h)=1$, then $v\mid\sigma(v)$. Since $v>1$, the integer $\sigma(v)/v$ is greater than $1$, and hence $v$ is multiperfect.
\end{proof}

The subcase $I(v)=I(v+h)$ is the equal-abundancy case. For parallel affine forms, Yamada's Theorem~1.1 leads to the same pair of linear prime forms $a_vL-1$ and $b_vL-1$; see
\cite[Theorem~1.1]{Yamada2017}.

The proof of Theorem~\ref{thm:regular-decomposition} uses the following three elementary lemmas and a two-linear-form sieve estimate.

\begin{lemma}\label{lem:factor-allocation}
Let $A,B,U,V,K$ be positive integers such that $(A,B)=1$ and $AV=BUK$. Then there exist positive integers $d_2,d_3,d_4,d_5$ satisfying
\[
A=d_2d_3,\qquad
U=d_2d_4,\qquad
V=Bd_4d_5,\qquad
K=d_3d_5.
\]
\end{lemma}
\begin{proof}
Since $(A,B)=1$, the relation $AV=BUK$ implies $B\mid V$. Write $V=BT$, and put
\[
d_2=(A,U),\qquad A=d_2d_3,\qquad U=d_2d_4.
\]
Then $(d_3,d_4)=1$, and cancellation gives $d_3T=d_4K$. Thus $d_3\mid K$; writing $K=d_3d_5$ gives $T=d_4d_5$, as required.
\end{proof}

\begin{lemma}\label{lem:diophantine}
Let $u,v,h$ be positive integers, and suppose
\begin{equation}\label{eq:diophantine}
Qv-Pu=h
\end{equation}
has an integer solution. Put $g=(u,v)$, $a=u/g$, and $b=v/g$. Then $g\mid h$, $(a,b)=1$, and every integer solution of \eqref{eq:diophantine} is of the form
\begin{equation}\label{eq:dioph-param}
P=P_0+b\lambda,
\qquad
Q=Q_0+a\lambda,
\qquad
\lambda\in\Z.
\end{equation}
If two parameters $\lambda,\lambda'$ both satisfy
\begin{equation}\label{eq:dioph-cong}
P\equiv Q\equiv-1\pmod d,
\end{equation}
then $\lambda\equiv\lambda'\pmod d$. Moreover, whenever \eqref{eq:dioph-cong} holds,
\begin{equation}\label{eq:d-divides-shift}
d\mid u-v-h.
\end{equation}
\end{lemma}
\begin{proof}
Since $g\mid Qv-Pu$, equation \eqref{eq:diophantine} gives $g\mid h$. Dividing by $g$, we obtain
\[
bQ-aP=\frac hg,
\]
with $(a,b)=1$, which gives \eqref{eq:dioph-param}.

If $\lambda$ and $\lambda'$ both satisfy \eqref{eq:dioph-cong}, subtraction gives
\[
d\mid b(\lambda-\lambda'),
\qquad
d\mid a(\lambda-\lambda').
\]
B\'ezout's identity and $(a,b)=1$ therefore imply $d\mid\lambda-\lambda'$. Finally, reducing \eqref{eq:diophantine} modulo $d$ under \eqref{eq:dioph-cong} gives $u-v\equiv h\pmod d$, proving \eqref{eq:d-divides-shift}.
\end{proof}

\begin{lemma}\label{lem:shifted-divisor-sum}
For every fixed $h\ge1$ and $X\ge2$,
\[
\sum_{\substack{uv\le X\\u-v-h\ne0}}
\tau(|u-v-h|)
\ll_h X(\log(2X))^2.
\]
\end{lemma}
\begin{proof}
\begin{sloppypar}
First suppose $u>v+h$. Then $v<\sqrt X$. For fixed $v$, the integer $m=u-v-h$ runs through a subinterval of $1\le m\le X/v$. Hence
\[
\sum_u\tau(u-v-h)\ll\frac Xv\log(2X),
\]
by the standard divisor estimate from Section~\ref{sec:preliminaries}. Summing over $v<\sqrt X$ gives $O(X(\log(2X))^2)$.
\end{sloppypar}

Now suppose $u<v+h$. If $u>2h$, then $v>u-h>u/2$, and therefore $u<\sqrt{2X}$. For fixed $u$, the positive integer $m=v+h-u$ satisfies $m\le X/u+h$, so
\[
\sum_v\tau(v+h-u)
\ll_h
\left(\frac Xu+h\right)\log(2X).
\]
Summation over $u\ll_h\sqrt X$ gives the required bound. The $O_h(1)$ values $u\le2h$ are handled directly by the same divisor estimate.
\end{proof}

\begin{lemma}\label{lem:two-linear-sieve}
Let $A\ne B$ be coprime positive integers and $T\ge3$. Then
\begin{equation}\label{eq:sieve-general}
\#\{L\le T:AL-1\ \text{and}\ BL-1\ \text{are prime}\}
\ll
\frac{T}{(\log T)^2}
\prod_{p\mid AB(A-B)}\left(1-\frac1p\right)^{-1}.
\end{equation}
Consequently, if $g=(v,h)$,
\[
A=\frac vg,\qquad B=\frac{v+h}{g},
\]
then
\begin{equation}\label{eq:sieve-special}
\#\{L\le T:AL-1,\ BL-1\ \text{are prime}\}
\ll_h
\frac{T\log\log(3v)}{(\log T)^2}.
\end{equation}
\end{lemma}
\begin{proof}
For a prime $p$, let
\[
\rho(p)=\#\{L\bmod p:(AL-1)(BL-1)\equiv0\pmod p\}.
\]
If $p\nmid AB(A-B)$, both forms have distinct roots modulo $p$, so $\rho(p)=2$. If $p\mid A$ or $p\mid B$, exactly one form has a root; if $p\mid A-B$ and $p\nmid AB$, the two roots coincide. Thus
\begin{equation}\label{eq:rho}
\rho(p)=
\begin{cases}
2,&p\nmid AB(A-B),\\
1,&p\mid AB(A-B).
\end{cases}
\end{equation}
Moreover $2\mid AB(A-B)$, so $\rho(p)<p$ for every prime $p$.

Apply the two-dimensional Selberg upper-bound sieve from Section~\ref{sec:preliminaries} with sifting level $T^{1/3}$. Apart from $O(T^{1/3})$ values for which one of the prime forms is itself below the sifting level, this gives
\[
\ll T\prod_{p<T^{1/3}}\left(1-\frac{\rho(p)}p\right).
\]
For $p\nmid AB(A-B)$,
\[
\frac{1-2/p}{(1-1/p)^2}=1-\frac1{(p-1)^2},
\]
and the corresponding product is bounded. At an exceptional prime the relative factor is $(1-1/p)^{-1}$. Mertens' estimate therefore gives \eqref{eq:sieve-general}; the $O(T^{1/3})$ contribution is absorbed by its right-hand side.

For the specialization in \eqref{eq:sieve-special},
\[
A-B=-\frac hg.
\]
Hence the factors contributed by primes dividing $A-B$ are $O_h(1)$, while
\[
\prod_{p\mid AB}\left(1-\frac1p\right)^{-1}
=\frac{AB}{\varphi(AB)}
\ll_h\log\log(3v).
\]
Substitution into \eqref{eq:sieve-general} proves \eqref{eq:sieve-special}.
\end{proof}

\subsection{The fixed-shift regular decomposition}

\begin{theorem}\label{thm:regular-decomposition}
For every fixed integer $h\ge1$,
\[
\mathcal R_h\subseteq\mathcal B_h,
\qquad
R_h(x)\ll_h\frac{x}{(\log x)^2}.
\]
Moreover, as $x\to\infty$,
\begin{equation}\label{eq:regular-remainder}
0\le B_h(x)-R_h(x)
\le
x\exp\!\left(
-\left(\frac1{\sqrt6}+o_h(1)\right)
\sqrt{\log x\,\log\log x}
\right).
\end{equation}
\end{theorem}

The proof below first removes the cases in which $n$ or $n+h$ is $y$-smooth, together with a large-prime-square exceptional set. Writing
\[
P=P^+(n),\qquad Q=P^+(n+h),\qquad d=(P+1,Q+1),
\]
we then split according to $d>z$ or $d\le z$. In the range $d>z$, Lemma~\ref{lem:diophantine} isolates the unique resonance $u=v+h$, which produces the regular family $\mathcal R_h$; the remaining configurations are estimated directly. The range $d\le z$ is controlled using Lemma~\ref{lem:factor-allocation} and the average estimate for $\tau(\sigma(n))$, after which the parameters $w,z,y$ are chosen to balance the resulting error terms.

\begin{proof}
Let $x$ be sufficiently large in terms of $h$, and introduce parameters
\[
2h<z<y<x,\qquad 1<w<z,
\]
to be chosen below.

We first discard those $n\le x$ for which either $P^+(n)\le y$ or $P^+(n+h)\le y$. By \eqref{eq:smooth}, their number is
\begin{equation}\label{eq:smooth-pruned}
\ll_h x\exp\{-(1+o(1))u\log u\},
\qquad
u=\frac{\log x}{\log y}.
\end{equation}
We next discard the integers for which $p^2\mid n(n+h)$ for some prime $p>y$. Since $p>h$, such a prime cannot divide both $n$ and $n+h$. For fixed $p$, there are at most two relevant residue classes modulo $p^2$. If either class occurs then $p^2\le x+h$, so its contribution is $O_h(x/p^2)$. The total number removed is therefore $O_h(x/y)$.

For every remaining $n$, put
\[
P=P^+(n),\qquad Q=P^+(n+h),
\]
and write
\[
n=Pu,\qquad n+h=Qv.
\]
Then $P,Q>y>h$. Both occur to the first power, so
\[
(P,u)=(Q,v)=1,
\qquad
P>P^+(u),\quad Q>P^+(v).
\]
Moreover, $P\ne Q$, since otherwise $P\mid h$. Hence
\begin{equation}\label{eq:main-dioph}
Qv-Pu=h.
\end{equation}
If $n\in\mathcal B_h$, write
\[
\sigma(n+h)=k\sigma(n),
\qquad k\in\N.
\]
By multiplicativity,
\begin{equation}\label{eq:main-sigma}
(Q+1)\sigma(v)=k(P+1)\sigma(u).
\end{equation}
Using $\sigma(n)\ge n$ and \eqref{eq:abundancy-bound}, we have
\begin{equation}\label{eq:k-bound}
\begin{aligned}
k
&=\frac{\sigma(n+h)}{\sigma(n)}
\le\frac{\sigma(n+h)}{n}
=\frac{n+h}{n}\,I(n+h)
\ll_h\log\log x.
\end{aligned}
\end{equation}

Put
\[
d=(P+1,Q+1),\qquad A=\frac{Q+1}{d},\qquad B=\frac{P+1}{d}.
\]
Then $(A,B)=1$, and \eqref{eq:main-sigma} becomes $A\sigma(v)=B\sigma(u)k$. Lemma~\ref{lem:factor-allocation} gives positive integers $d_2,d_3,d_4,d_5$ such that
\begin{equation}\label{eq:factor-allocation}
A=d_2d_3,\qquad
\sigma(u)=d_2d_4,\qquad
\sigma(v)=Bd_4d_5,\qquad
k=d_3d_5.
\end{equation}
In particular, $d_3,d_5\ll_h\log\log x$.

We first treat $d>z$. Since $PQuv=n(n+h)\le x(x+h)\le2x^2$, either $PQ\le xw$ or $uv\le2x/w$.

Suppose $PQ\le xw$. For fixed distinct primes $P,Q$, the two congruences
\[
n\equiv0\pmod P,\qquad n\equiv-h\pmod Q
\]
determine one residue class modulo $PQ$. Thus the number of $n\le x$ is
\[
\ll1+\frac{x}{PQ}\ll\frac{xw}{PQ}.
\]
Write $P=d\ell-1$ and $Q=dm-1$. Since $P,Q>y$,
\[
\frac1P\ll\frac1{d\ell},\qquad \frac1Q\ll\frac1{dm}.
\]
Summing without imposing primality gives
\begin{equation}\label{eq:large-d-PQ}
\ll xw\sum_{d>z}\frac1{d^2}
\left(\sum_{\ell\ll x}\frac1\ell\right)
\left(\sum_{m\ll x}\frac1m\right)
\ll\frac{xw(\log x)^2}{z}.
\end{equation}

We may therefore assume $uv\le2x/w$. Apply Lemma~\ref{lem:diophantine} to \eqref{eq:main-dioph}. For fixed $u,v$, put $g=(u,v)$, $a=u/g$, $b=v/g$. Since $g\mid h$, the parameter $\lambda$ in \eqref{eq:dioph-param} is restricted by $P\ll x/u$ and $Q\ll_h x/v$ to an interval of length $O_h(x/(uv))$. The congruences $P\equiv Q\equiv-1\pmod d$ restrict $\lambda$ to at most one residue class modulo $d$. Consequently, for fixed $u,v,d$, there are
\begin{equation}\label{eq:lambda-count}
\ll_h1+\frac{x}{duv}
\end{equation}
possibilities. Moreover,
\begin{equation}\label{eq:d-divides-delta}
d\mid u-v-h.
\end{equation}

The unique resonant case in \eqref{eq:d-divides-delta} is
\begin{equation}\label{eq:resonance}
u=v+h.
\end{equation}
Put $g=(v,h)$, $a=v/g$, and $b=(v+h)/g$. Then \eqref{eq:main-dioph} becomes
\[
aQ-bP=b-a,
\]
or $a(Q+1)=b(P+1)$. Since $(a,b)=1$, there is $L\ge1$ such that
\begin{equation}\label{eq:resonant-primes}
P=aL-1,\qquad Q=bL-1.
\end{equation}
Equation \eqref{eq:main-sigma} now gives
\begin{equation}\label{eq:resonant-sigma}
(v+h)\sigma(v)=k\,v\sigma(v+h).
\end{equation}
Thus $v\in\mathcal V_h$. The case $v=1$ cannot occur, since \eqref{eq:resonant-sigma} would imply $h+1=k\sigma(h+1)$.

Since $P$ and $Q$ occur to the first power as the largest prime factors of $n$ and $n+h$, respectively,
\[
(P,v+h)=1,\qquad (Q,v)=1.
\]
Thus every resonant solution surviving the preceding exclusions belongs to $\mathcal R_h$.

Conversely, suppose $n\in\mathcal R_h$. From \eqref{eq:regular-n},
\[
n+h=vq_{v,L}.
\]
Using \eqref{eq:cross-coprime} and multiplicativity,
\[
\sigma(n)=a_vL\,\sigma(v+h),
\qquad
\sigma(n+h)=b_vL\,\sigma(v).
\]
Therefore
\[
\frac{\sigma(n+h)}{\sigma(n)}
=
\frac{(v+h)\sigma(v)}{v\sigma(v+h)}
\in\N,
\]
and hence $\mathcal R_h\subseteq\mathcal B_h$.

We next estimate $R_h(x)$. For fixed $v\in\mathcal V_h$, condition \eqref{eq:regular-n} gives
\begin{equation}\label{eq:L-bound}
L\le\frac{\delta_vx}{v(v+h)}+\frac{\delta_v}{v}
\ll_h\frac{x}{v(v+h)}+1.
\end{equation}
Every relevant $v$ satisfies $v\ll_h\sqrt x$. Indeed, if $v>2h$, then $\delta_v\le h$, and therefore
\[
p_{v,L}\ge a_v-1\ge\frac vh-1\ge\frac{v}{2h}.
\]
Thus $x\ge(v+h)p_{v,L}\gg_hv^2$.

For $v\le x^{1/4}$, the upper bound in \eqref{eq:L-bound} is $\gg_hx^{1/2}$, so its logarithm is comparable with $\log x$. Lemma~\ref{lem:two-linear-sieve} therefore gives
\[
\ll_h
\frac{\log\log(3v)}{(\log x)^2}
\left(\frac{x}{v(v+h)}+1\right)
\]
possible values of $L$. Summing over $v\le x^{1/4}$ and using
\[
\sum_{v\ge2}\frac{\log\log(3v)}{v^2}<\infty
\]
gives $O_h(x/(\log x)^2)$. For $x^{1/4}<v\ll_h\sqrt x$, we ignore primality and use \eqref{eq:L-bound}, obtaining
\[
\ll_h x\sum_{v>x^{1/4}}\frac1{v^2}+\sqrt x\ll_hx^{3/4}.
\]
Hence
\begin{equation}\label{eq:Rh-bound}
R_h(x)\ll_h\frac{x}{(\log x)^2}.
\end{equation}

It remains to estimate the nonresonant part of the range $d>z$. Put $\Delta=u-v-h\ne0$.

First suppose $duv\le x$. The $1$ in \eqref{eq:lambda-count} is then absorbed, and the number of parameters is $O_h(x/(duv))$. Since $d\mid\Delta$, write $|\Delta|=jd$, $j\ge1$.
If $\Delta>0$, then $u=v+h+jd\ge jd$, whence
\[
\frac{x}{duv}\le\frac{x}{jd^2v}.
\]
Summing over $d>z$, $j$, and $v$ gives $\ll_hx(\log x)^2/z$. If $\Delta<0$, then $v=u+jd-h$. Since $d>z>2h$, we have $v\ge jd/2$, and the same argument again gives $O_h(x(\log x)^2/z)$.

If $duv>x$, \eqref{eq:lambda-count} gives $O_h(1)$ possibilities for every triple $u,v,d$. For fixed $u,v$, the number of admissible $d$'s is at most $\tau(|u-v-h|)$. Since $uv\le2x/w$, Lemma~\ref{lem:shifted-divisor-sum} gives a total contribution $\ll_hx(\log x)^2/w$. Together with \eqref{eq:large-d-PQ}, the nonregular contribution from $d>z$ is therefore
\begin{equation}\label{eq:large-d-nonregular}
\ll_h
\frac{xw(\log x)^2}{z}
+
\frac{x(\log x)^2}{z}
+
\frac{x(\log x)^2}{w}.
\end{equation}

We turn to the range $d\le z$. Since
\[
(Pv)(Qu)=PQuv=n(n+h)\le2x^2,
\]
we have
\[
\min\{Pv,Qu\}\le\sqrt2\,x\ll x.
\]
Thus every such solution lies in at least one of the two subranges
\[
Pv\ll x
\qquad\text{or}\qquad
Qu\ll x.
\]

We first bound the contribution from the subrange $Pv\ll x$. Then $(P,v)=1$; otherwise $P\mid n$ and $P\mid n+h$, contrary to $P>h$. For fixed $P,v$, the solutions of \eqref{eq:main-dioph} are parametrized by
\[
Q=Q_0+P\nu,\qquad u=u_0+v\nu,
\]
and $Pu\le x$ gives $O_h(x/(Pv))$ relevant parameters. Put
\[
D=\frac{P+1}{d}=B.
\]
Then $P=dD-1\asymp dD$. By \eqref{eq:factor-allocation},
$D\mid\sigma(v)$, while $P>y$ and $d\le z$ give $D>y/z$. Since $P=dD-1$ is determined by $d$ and $D$, this contribution is
\[
\ll_h
x\sum_{v\le2x}\frac1v
\sum_{d\le z}\frac1d
\sum_{\substack{D\mid\sigma(v)\\D>y/z}}\frac1D.
\]
The innermost sum is at most $(z/y)\tau(\sigma(v))$. Hence this range contributes
\begin{equation}\label{eq:small-d-first}
\ll_h
\frac{xz\log x}{y}
\sum_{v\le2x}\frac{\tau(\sigma(v))}{v}.
\end{equation}

We next bound the contribution from the subrange $Qu\ll x$. Then $(Q,u)=1$. For fixed $Q,u$,
\[
P=P_0+Q\nu,\qquad v=v_0+u\nu,
\]
and the parameter count is $O_h(x/[u(Q+1)])$. By \eqref{eq:factor-allocation},
\[
\frac{Q+1}{d}=d_2d_3,
\qquad\text{so that}\qquad
Q+1=d\,d_2d_3.
\]
Put
\[
D=\frac{Q+1}{dd_3}=d_2.
\]
Then $D\mid\sigma(u)$ and $D>y/(zd_3)$. Thus the contribution is
\begin{equation}\label{eq:small-d-second}
\ll_h
x\sum_{u\le2x}\frac1u
\sum_{d\le z}\frac1d
\sum_{d_3\ll_h\log\log x}\frac1{d_3}
\sum_{\substack{D\mid\sigma(u)\\D>y/(zd_3)}}\frac1D.
\end{equation}
The innermost sum is at most $(zd_3/y)\tau(\sigma(u))$, which cancels the factor $1/d_3$ in \eqref{eq:small-d-second}. Since there are $O_h(\log\log x)$ possible positive values of $d_3$, \eqref{eq:small-d-second} is
\begin{equation}\label{eq:small-d-second-simplified}
\ll_h
\frac{xz(\log x)(\log\log x)}{y}
\sum_{u\le2x}\frac{\tau(\sigma(u))}{u}.
\end{equation}
Using \eqref{eq:partial-sum-tau-sigma} in \eqref{eq:small-d-first} and \eqref{eq:small-d-second-simplified}, and absorbing powers of $\log x$ and $\log\log x$ into the exponential, we obtain
\begin{equation}\label{eq:small-d-total}
\#\{n\le x:d\le z\}
\ll_h
\frac{xz}{y}
\exp\!\left((C_0+o(1))\sqrt{\frac{\log x}{\log\log x}}\right).
\end{equation}

We choose $w$ and $z$ to balance the leading algebraic losses in
\eqref{eq:large-d-nonregular} and \eqref{eq:small-d-total}: first
$w/z=1/w$, and then $z^{-1/2}=z/y$. Thus
\[
w=z^{1/2},\qquad z=y^{2/3}.
\]
Combining \eqref{eq:smooth-pruned}, the large-prime-square estimate, \eqref{eq:large-d-nonregular}, and \eqref{eq:small-d-total} gives
\begin{equation}\label{eq:aggregate}
\begin{aligned}
B_h(x)-R_h(x)
\ll_h{}&
x\exp\{-(1+o(1))u\log u\}
+\frac xy
+\frac{x(\log x)^2}{y^{1/3}}\\
&+\frac{x}{y^{1/3}}
\exp\!\left((C_0+o(1))\sqrt{\frac{\log x}{\log\log x}}\right).
\end{aligned}
\end{equation}
Put $L_1=\log x$, $L_2=\log\log x$, and choose
\[
\log y=\sqrt{\frac32L_1L_2}.
\]
Then
\[
u=\sqrt{\frac23\,\frac{L_1}{L_2}},
\qquad
u\log u=
\left(\frac1{\sqrt6}+o(1)\right)\sqrt{L_1L_2},
\]
whereas
\[
\frac13\log y=\frac1{\sqrt6}\sqrt{L_1L_2}.
\]
Finally,
\[
\sqrt{\frac{L_1}{L_2}}=o(\sqrt{L_1L_2}),
\]
so the exponential factor in \eqref{eq:small-d-total} affects only the lower-order term in the exponent. Thus \eqref{eq:aggregate} yields \eqref{eq:regular-remainder}. Together with \eqref{eq:Rh-bound}, this completes the proof.
\end{proof}

\subsection{Consequences and affine equalities}

\begin{corollary}\label{cor:divisibility-bound}
For every fixed $h\ge1$,
\[
B_h(x)\ll_h\frac{x}{(\log x)^2}.
\]
\end{corollary}
\begin{proof}
By Theorem~\ref{thm:regular-decomposition},
\[
B_h(x)
\le
R_h(x)
+
x\exp\!\left(
-\left(\frac1{\sqrt6}+o_h(1)\right)
\sqrt{\log x\,\log\log x}
\right).
\]
The second term is $o(x/(\log x)^2)$, while $R_h(x)\ll_hx/(\log x)^2$.
\end{proof}

\begin{corollary}\label{cor:reciprocal-sum}
For every fixed $h\ge1$, $\sum_{n\in\mathcal B_h}\frac1n<\infty$.
Consequently, for every $k\in\N$,
\[
\sum_{\substack{n\ge1\\ \sigma(n+h)=k\sigma(n)}}\frac1n<\infty.
\]
\end{corollary}
\begin{proof}
By Corollary~\ref{cor:divisibility-bound}, partial summation gives, for $X\ge3$,
\[
\sum_{\substack{n\le X\\ n\in\mathcal B_h}}\frac1n
=
\frac{B_h(X)}{X}
+
\int_1^X\frac{B_h(t)}{t^2}\,dt
\ll_h
1+\int_3^\infty\frac{dt}{t(\log t)^2}
<\infty.
\]
Letting $X\to\infty$ proves the first assertion. The second follows from
\[
\{n\in\N:\sigma(n+h)=k\sigma(n)\}\subseteq\mathcal B_h.
\]
\end{proof}

For $h\in\Z\setminus\{0\}$ and $\lambda>0$, define
\[
\mathcal A_{h,\lambda}
=
\{n\in\N:n+h\ge1,\ \sigma(n+h)=\lambda\sigma(n)\},
\]
and
\[
A_{h,\lambda}(x)=\#(\mathcal A_{h,\lambda}\cap[1,x]).
\]
This notation will also be used in Section~\ref{sec:levelsets}.

\begin{corollary}\label{cor:integer-multiplier}
For every fixed $h\ge1$,
\[
\sup_{k\in\N}A_{h,k}(x)\ll_h\frac{x}{(\log x)^2}.
\]
\end{corollary}
\begin{proof}
If $k\in\N$ and $\sigma(n+h)=k\sigma(n)$, then $\sigma(n)\mid\sigma(n+h)$. Hence
\[
A_{h,k}(x)\le B_h(x)
\]
for every $k\in\N$. The bound in Corollary~\ref{cor:divisibility-bound} is independent of $k$, so taking the supremum over $k\in\N$ gives the result.
\end{proof}

\begin{corollary}\label{cor:dichotomy}
Let $h\ge1$ be fixed.

\textup{(i)} If $\mathcal V_h=\varnothing$, then
\[
B_h(x)
\le
x\exp\!\left(
-\left(\frac1{\sqrt6}+o_h(1)\right)
\sqrt{\log x\,\log\log x}
\right).
\]

\textup{(ii)} If $\mathcal V_h\ne\varnothing$, then, assuming the Bateman--Horn conjecture,
\[
B_h(x)\asymp_h\frac{x}{(\log x)^2}.
\]
\end{corollary}
\begin{proof}
If $\mathcal V_h=\varnothing$, then $\mathcal R_h=\varnothing$, and part~(i) follows immediately from Theorem~\ref{thm:regular-decomposition}.

For part~(ii), fix once and for all $v_0\in\mathcal V_h$, and put
\[
g_0=(v_0,h),\qquad a_0=\frac{v_0}{g_0},\qquad b_0=\frac{v_0+h}{g_0}.
\]
Consider
\[
f_1(L)=a_0L-1,\qquad f_2(L)=b_0L-1.
\]
These two linear polynomials form an admissible system. Indeed, for a prime $p$, the number $\rho(p)$ of roots of $f_1f_2$ modulo $p$ equals $1$ if $p\mid a_0b_0(b_0-a_0)$ and $2$ otherwise. Since $a_0b_0(b_0-a_0)$ is even, $\rho(2)=1<2$, while $\rho(p)\le2<p$ for every $p\ge3$.

The Bateman--Horn conjecture therefore gives
\[
\#\{L\le T:f_1(L),f_2(L)\ \text{prime}\}
\sim
\mathfrak S_{v_0,h}\frac{T}{(\log T)^2}
\]
for some $\mathfrak S_{v_0,h}>0$. For sufficiently large $L$, both primes exceed the fixed integers $v_0+h$ and $v_0$, so the two coprimality conditions in \eqref{eq:cross-coprime} hold automatically. Since $n=(v_0+h)(a_0L-1)\le x$ permits $L\asymp_{v_0,h}x$, we obtain
\[
R_h(x)\gg_{v_0,h}\frac{x}{(\log x)^2}.
\]
Having fixed $v_0$ for this $h$, we may absorb the dependence on $v_0$
into the $h$-dependent implied constant. Hence
\[
B_h(x)\gg_h\frac{x}{(\log x)^2}.
\]
The reverse inequality follows from Corollary~\ref{cor:divisibility-bound}.
\end{proof}

We use the Bateman--Horn conjecture in its usual prime-values form
\cite{BatemanHorn1962}.

\begin{remark}\label{rem:V22}
The exceptional set need not be empty. Indeed,
\[
I(6)=I(28)=2,
\]
so
\[
6\in\mathcal V_{22}.
\]
For this pair, $\delta_6=2$, and the associated forms are $3L-1$ and $14L-1$. The corresponding shifted-equality family $n=28(3L-1)$ also occurs in Yamada's construction
\cite{Yamada2017}.
\end{remark}

Let $q\ge1$ and let $r_1,r_2\in\Z$ be distinct. For $x\ge1$, define
\[
N_{q,r_1,r_2}(x)
=
\#\left\{
 n\in\N:
 \begin{array}{l}
 n\le x,\\
 qn+r_1\ge1,\ qn+r_2\ge1,\\
 \sigma(qn+r_1)=\sigma(qn+r_2)
 \end{array}
\right\}.
\]

\begin{corollary}\label{cor:parallel}
Put $h=|r_1-r_2|$. Then
\[
N_{q,r_1,r_2}(x)\ll_{q,h}\frac{x}{(\log x)^2}.
\]
The implied multiplicative constant may be chosen to depend only on $q$ and $h$; the threshold may also depend on $\min(r_1,r_2)$.
\end{corollary}
\begin{proof}
Put $s=\min(r_1,r_2)$. The two affine arguments are $qn+s$ and $qn+s+h$, and both are positive precisely when $qn+s\ge1$. For every admissible $n$, set $m=qn+s$. Then $m\ge1$, and
\[
\{qn+r_1,qn+r_2\}=\{m,m+h\}.
\]
Hence $\sigma(m)=\sigma(m+h)$, and in particular $m\in\mathcal B_h$. The map $n\mapsto m$ is injective, and $n\le x$ implies $m\le q\lfloor x\rfloor+s$. Thus, whenever the argument on the right is positive,
\begin{equation}\label{eq:parallel-reduction}
N_{q,r_1,r_2}(x)\le B_h(q\lfloor x\rfloor+s).
\end{equation}
If that argument is nonpositive, the left-hand side is zero. Corollary~\ref{cor:divisibility-bound} applied to \eqref{eq:parallel-reduction}, together with $q\lfloor x\rfloor+s=qx+O_{q,s}(1)$, gives the result.
\end{proof}

For comparison, Yamada's Theorems~1.1--1.2 distinguish solutions arising
from an explicit prime-form construction from the remaining solutions
\cite{Yamada2017}. In the parallel case, Corollary~\ref{cor:parallel}
is consistent with this framework and follows here directly from
Corollary~\ref{cor:divisibility-bound}.
\section{Uniform bounds for shifted proportionality relations}
\label{sec:levelsets}

Recall from Section~\ref{sec:divisibility} that, for $h\in\Z\setminus\{0\}$ and $\lambda>0$,
\[
\mathcal A_{h,\lambda}
=
\{n\in\N:n+h\ge1,\ \sigma(n+h)=\lambda\sigma(n)\},
\]
and
\[
A_{h,\lambda}(x)=\#(\mathcal A_{h,\lambda}\cap[1,x]).
\]
We prove a bound that is uniform in $\lambda$. The argument
truncates the additive function $\log(\sigma(n)/n)$, exploits the
periodicity of the truncated difference, and identifies its distribution
on a complete period exactly with a product distribution furnished by
the Chinese remainder theorem. The resulting local variables are then
treated using the concentration inequality of Theorem~\ref{thm:KR}.

\begin{theorem}\label{thm:uniform-levelset}
There exists an absolute constant $C>0$ such that, for every nonzero integer $h$, there exists $x_0(h)$ with the property that,
for every $x\ge x_0(h)$,
\[
\sup_{\lambda>0}A_{h,\lambda}(x)
\le
C\frac{x}{\sqrt{\log\log\log x}}.
\]
Consequently, for every fixed $h\ne0$,
\[
\sup_{\lambda>0}\frac{A_{h,\lambda}(x)}x\longrightarrow0
\qquad(x\to\infty).
\]
The constant $C$ is absolute; only the threshold $x_0(h)$ may depend on the shift. The estimate therefore holds simultaneously for all real $\lambda>0$.
\end{theorem}

\begin{remark}\label{rem:scope-levelsets}
For $h>0$ and $k\in\N$, Corollary~\ref{cor:integer-multiplier} gives the stronger estimate
\[
A_{h,k}(x)\ll_h\frac{x}{(\log x)^2}.
\]
Theorem~\ref{thm:uniform-levelset} has a different scope: it applies to every fixed nonzero integer shift, including negative shifts, and simultaneously to all real $\lambda>0$. Its proof is independent of the divisibility argument of Section~\ref{sec:divisibility}.
\end{remark}

\subsection{Truncation and exceptional sets}

Set
\[
g(m)=\log I(m)=\log\frac{\sigma(m)}m.
\]
For a prime power $p^a$, $a\ge1$, define
\[
\ell(p^a)
=
\log\left(1+\frac1p+\cdots+\frac1{p^a}\right)
=
\log\left(\frac{1-p^{-(a+1)}}{1-p^{-1}}\right),
\]
and put $\ell(1)=0$. Since
\[
\frac{\sigma(m)}m
=
\prod_{p^a\parallel m}\left(1+\frac1p+\cdots+\frac1{p^a}\right),
\]
we have
\[
g(m)=\sum_p\ell\!\left(p^{v_p(m)}\right).
\]
For every prime $p$ and $a\ge1$,
\[
0\le\ell(p^a)\le-\log(1-p^{-1})\le\frac2p,
\]
whereas
\[
\ell(p)=\log(1+p^{-1})\ge\frac1{2p}.
\]

Let $y\ge2$ and let $r\ge1$ be an integer. Define
\[
g_{y,r}(m)
=
\sum_{p\le y}\ell\!\left(p^{\min(v_p(m),r)}\right)
\]
and
\[
T_y(m)=\sum_{p>y}\ell\!\left(p^{v_p(m)}\right).
\]
If $v_p(m)\le r$ for every $p\le y$, then
\[
g(m)=g_{y,r}(m)+T_y(m).
\]
We first estimate the contribution from primes exceeding $y$.

\begin{lemma}\label{lem:large-prime-tail}
Uniformly for $X\ge1$ and $y\ge2$,
\[
\sum_{m\le X}T_y(m)\ll\frac{X}{y\log y},
\]
with an absolute implied constant.
\end{lemma}
\begin{proof}
By nonnegativity and the preceding bound for $\ell(p^a)$,
\[
\begin{aligned}
\sum_{m\le X}T_y(m)
&\le
\sum_{p>y}\sum_{a\ge1}\frac2p\,\#\{m\le X:v_p(m)=a\}\\
&\le
2X\sum_{p>y}\sum_{a\ge1}\frac1{p^{a+1}}\\
&=
2X\sum_{p>y}\frac1{p(p-1)}
\ll
X\sum_{p>y}\frac1{p^2}.
\end{aligned}
\]
The prime-square tail estimate from
Section~\ref{sec:preliminaries} completes the proof.
\end{proof}

Fix henceforth
\[
h\in\Z\setminus\{0\},\qquad H=|h|.
\]
We call $n$ admissible if $n\in\N$, $n\le x$, and $n+h\ge1$.
For such $n$, the identity
\[
\sigma(m)=m\exp(g(m))
\]
shows that
\[
\sigma(n+h)=\lambda\sigma(n)
\]
is equivalent to
\[
g(n+h)-g(n)=\log\lambda-\log\left(1+\frac hn\right).
\]
Put
\[
D^{(h)}_{y,r}(n)=g_{y,r}(n+h)-g_{y,r}(n).
\]
Let $0<\varepsilon\le1$. Define
\[
E_0
=
\left\{n\le x:n+h\ge1,\ \left|\log\left(1+\frac hn\right)\right|>\varepsilon\right\},
\]
\[
E_1
=
\left\{n\le x:n+h\ge1,\ T_y(n)>\varepsilon\ \text{or}\ T_y(n+h)>\varepsilon\right\},
\]
and
\[
E_2
=
\left\{n\le x:n+h\ge1,\ v_p(n)>r\ \text{or}\ v_p(n+h)>r\ \text{for some }p\le y\right\}.
\]

\begin{lemma}\label{lem:exceptional-sets}
With the notation above,
\[
\#E_0\ll\frac{H}{\varepsilon},
\qquad
\#E_1\ll\frac{x+H}{\varepsilon y\log y},
\qquad
\#E_2\ll(x+H)2^{-r},
\]
with absolute implied constants.
\end{lemma}
\begin{proof}
Suppose first that $h>0$. Since
\[
0\le\log\left(1+\frac hn\right)\le\frac Hn,
\]
membership in $E_0$ implies $n<H/\varepsilon$.

Now let $h=-H<0$. The admissibility condition $n+h\ge1$ implies $n>H$. There are $O(H)$ integers with $H<n\le2H$, while for $n>2H$,
\[
-\log\left(1-\frac Hn\right)
\le
\frac{H/n}{1-H/n}
\le
\frac{2H}{n}.
\]
Since $\varepsilon\le1$, these estimates give $\#E_0\ll H/\varepsilon$.

For $E_1$, both positive arguments $n$ and $n+h$ are at most $x+H$. Lemma~\ref{lem:large-prime-tail} and Markov's inequality therefore give
\[
\#E_1\ll\frac{x+H}{\varepsilon y\log y}.
\]
Finally, if $v_p(m)>r$, then $p^{r+1}\mid m$. Hence
\[
\#E_2
\le
(2x+H)\sum_{p\le y}p^{-(r+1)}
\ll
(x+H)2^{-r},
\]
because
\[
\sum_{p\le y}p^{-(r+1)}\le\sum_{m\ge2}m^{-(r+1)}\ll2^{-r}.
\]
\end{proof}

For $n\le x$ with $n+h\ge1$ and
$n\notin E_0\cup E_1\cup E_2$, every solution of
$\sigma(n+h)=\lambda\sigma(n)$ satisfies
\begin{equation}\label{eq:truncated-window}
\left|D^{(h)}_{y,r}(n)-\log\lambda\right|\le3\varepsilon.
\end{equation}

\subsection{Exact CRT model and anti-concentration}

Set
\[
M_{y,r}=\prod_{p\le y}p^r.
\]
For $u\in\Z/p^r\Z$, define the clipped valuation
\[
\nu_{p,r}(u)=\max\{0\le a\le r:u\equiv0\pmod{p^a}\}.
\]
Thus $\nu_{p,r}(0)=r$, and, for every positive integer $m$,
\[
\nu_{p,r}(m\bmod p^r)=\min(v_p(m),r).
\]
It follows that $D^{(h)}_{y,r}(n)$ is periodic modulo $M_{y,r}$.

Let $U$ be uniformly distributed on $\Z/M_{y,r}\Z$. By the Chinese remainder theorem, the uniform measure on this finite ring is exactly the product of the uniform measures on
\[
\prod_{p\le y}\Z/p^r\Z.
\]
Accordingly, for $p\le y$, let $U_p$ be uniformly distributed on $\Z/p^r\Z$, with the family $(U_p)_{p\le y}$ mutually independent. Define
\[
Z^{(h)}_{p,r}
=
\ell\!\left(p^{\nu_{p,r}(U_p+h)}\right)
-
\ell\!\left(p^{\nu_{p,r}(U_p)}\right).
\]
Then the distribution of $D^{(h)}_{y,r}$ on every complete residue period is exactly the distribution of
\begin{equation}\label{eq:W-model}
W^{(h)}_{y,r}=\sum_{p\le y}Z^{(h)}_{p,r}.
\end{equation}
Thus all independence used below is exact: it comes directly from the Chinese remainder theorem on the finite period $M_{y,r}$.

We next determine the local law at primes not dividing the shift.

\begin{lemma}\label{lem:local-law}
Let $p\nmid h$. Then
\[
\Pr\!\left(Z^{(h)}_{p,r}=0\right)=1-\frac2p.
\]
For $1\le a<r$,
\[
\Pr\!\left(Z^{(h)}_{p,r}=\ell(p^a)\right)
=
\Pr\!\left(Z^{(h)}_{p,r}=-\ell(p^a)\right)
=
\frac{p-1}{p^{a+1}},
\]
while
\[
\Pr\!\left(Z^{(h)}_{p,r}=\ell(p^r)\right)
=
\Pr\!\left(Z^{(h)}_{p,r}=-\ell(p^r)\right)
=
\frac1{p^r}.
\]
In particular,
\[
\Pr\!\left(Z^{(h)}_{p,r}>0\right)
=
\Pr\!\left(Z^{(h)}_{p,r}<0\right)
=
\frac1p.
\]
If $p\ge5$ and $0<L<\ell(p)$, then
\begin{equation}\label{eq:local-concentration}
Q_L\!\left(Z^{(h)}_{p,r}\right)=1-\frac2p.
\end{equation}
\end{lemma}
\begin{proof}
Since $p\nmid h$, the residue classes $0$ and $-h$ modulo $p$ are distinct. With probability $1-2/p$, neither $U_p$ nor $U_p+h$ is divisible by $p$, and hence $Z^{(h)}_{p,r}=0$.

If $p\mid U_p+h$, then $p\nmid U_p$, so
$Z^{(h)}_{p,r}>0$. For $1\le a<r$,
\[
\Pr\!\left(\nu_{p,r}(U_p+h)=a\right)
=
\frac{p-1}{p^{a+1}},
\]
while
\[
\Pr\!\left(\nu_{p,r}(U_p+h)=r\right)
=
\frac1{p^r}.
\]
These are the positive atoms. The same counting argument applied to
$U_p$ gives the corresponding negative atoms.

Every nonzero atom has absolute value at least $\ell(p)$. Hence an interval of length $L<\ell(p)$ that contains zero contains no nonzero atom, and therefore has mass at most $1-2/p$, with equality for an interval containing zero.

Now consider an interval of length $L<\ell(p)$ that does not contain zero. It cannot meet both the positive and negative sign clusters. It may meet more than one atom within one sign cluster, but the total mass of that entire cluster is $1/p$. Hence its probability mass is at most $1/p$. Since $p\ge5$, $1/p<1-2/p$. Therefore the maximal interval mass is the mass of the zero atom, and \eqref{eq:local-concentration} follows.
\end{proof}

At primes dividing $h$, the local law is different. If $p\mid h$, the residue classes $0$ and $-h$ coincide modulo $p$; in particular, if $p^r\mid h$, then
\[
U_p+h\equiv U_p\pmod{p^r},
\]
and consequently $Z^{(h)}_{p,r}=0$ identically. These coordinates remain part of the complete random variable $W^{(h)}_{y,r}$, but they are omitted from the anti-concentration subsum.

Set
\[
z=\min\left(y,\frac1{13\varepsilon}\right)
\]
and define
\[
\mathcal P_h^\ast(z)=\{p\text{ prime}:5\le p\le z,\ p\nmid h\}.
\]
Put
\[
S_h^\ast=\sum_{p\in\mathcal P_h^\ast(z)}Z^{(h)}_{p,r}.
\]
The remaining coordinates are collected in $\widetilde W_h$, so that
\[
W^{(h)}_{y,r}=S_h^\ast+\widetilde W_h.
\]
Because the two sums depend on disjoint sets of CRT coordinates, $S_h^\ast$ and $\widetilde W_h$ are independent.

For $p\in\mathcal P_h^\ast(z)$,
\[
6\varepsilon<\frac1{2p}\le\ell(p),
\]
and Lemma~\ref{lem:local-law} therefore gives
\[
Q_{6\varepsilon}\!\left(Z^{(h)}_{p,r}\right)=1-\frac2p.
\]
Define
\[
H_h(z)=\sum_{\substack{5\le p\le z\\p\nmid h}}\frac1p,
\]
with the convention
\[
H_h(z)^{-1/2}=+\infty
\qquad\text{when }H_h(z)=0.
\]
If $\mathcal P_h^\ast(z)$ is nonempty, apply Theorem~\ref{thm:KR} to $S_h^\ast$ with global scale $L=6\varepsilon$ and with every local scale equal to $L_j=6\varepsilon$. Thus $0<L_j\le2L$, as required in Theorem~\ref{thm:KR}, and
\[
1-Q_{6\varepsilon}\!\left(Z^{(h)}_{p,r}\right)=\frac2p.
\]
The universal constant in Theorem~\ref{thm:KR} therefore gives
\[
Q_{6\varepsilon}(S_h^\ast)\ll H_h(z)^{-1/2}.
\]
By Lemma~\ref{lem:convolution} and the independent decomposition $W^{(h)}_{y,r}=S_h^\ast+\widetilde W_h$, we obtain
\[
Q_{6\varepsilon}\!\left(W^{(h)}_{y,r}\right)\le Q_{6\varepsilon}(S_h^\ast).
\]
If $\mathcal P_h^\ast(z)$ is empty, the trivial bound by $1$, together with the convention above, gives the same conclusion. Hence
\begin{equation}\label{eq:W-concentration}
Q_{6\varepsilon}\!\left(W^{(h)}_{y,r}\right)
\ll
\min\left\{1,H_h(z)^{-1/2}\right\}.
\end{equation}
The implied constant is absolute.

It remains to transfer the complete-period estimate to the original integers. The admissible set
\[
\{n\in\N:n\le x,\ n+h\ge1\}
\]
is a consecutive interval. Partition it into complete blocks of length $M_{y,r}$ and at most one incomplete final block. For every $t\in\mathbb R$,
\begin{equation}\label{eq:period-transfer}
\begin{aligned}
&\#\left\{n\le x:n+h\ge1,\ \left|D^{(h)}_{y,r}(n)-t\right|\le3\varepsilon\right\}\\
&\qquad\le
x\,\Pr\!\left(\left|W^{(h)}_{y,r}-t\right|\le3\varepsilon\right)+M_{y,r}.
\end{aligned}
\end{equation}
The event inside the probability is an interval of length $6\varepsilon$. Combining \eqref{eq:W-concentration} and \eqref{eq:period-transfer}, and then taking $t=\log\lambda$, gives the required estimate for the nonexceptional solutions.

\subsection{Master estimate and parameter choice}

\begin{proposition}\label{prop:master}
Let
\[
x\ge1,\qquad
y\ge5,\qquad
r\ge1,\qquad
0<\varepsilon\le1,
\]
and put
\[
z=\min\left(y,\frac1{13\varepsilon}\right).
\]
Then, uniformly for all $\lambda>0$,
\begin{equation}\label{eq:master}
\frac{A_{h,\lambda}(x)}{x}
\ll
\frac{H}{x\varepsilon}
+
\frac{1+H/x}{\varepsilon y\log y}
+
(1+H/x)2^{-r}
+
\frac{M_{y,r}}{x}
+
\min\left\{1,H_h(z)^{-1/2}\right\},
\end{equation}
where $H=|h|$. The implied constant is absolute.
\end{proposition}
\begin{proof}
The exceptional sets in Lemma~\ref{lem:exceptional-sets} contribute, after division by $x$,
\[
\ll\frac{H}{x\varepsilon},
\qquad
\ll\frac{1+H/x}{\varepsilon y\log y},
\qquad
\ll(1+H/x)2^{-r},
\]
respectively. Every $n\le x$ counted by $A_{h,\lambda}(x)$ and not belonging to
$E_0\cup E_1\cup E_2$ satisfies \eqref{eq:truncated-window}. Applying \eqref{eq:period-transfer} with $t=\log\lambda$, and then \eqref{eq:W-concentration}, contributes
\[
\frac{M_{y,r}}x
+
O\!\left(\min\left\{1,H_h(z)^{-1/2}\right\}\right).
\]
Combining these five contributions proves \eqref{eq:master}.
\end{proof}

We now choose the parameters. Put
\[
L_3=\log\log\log x
\]
and, for sufficiently large $x$, take
\begin{equation}\label{eq:parameters}
r=\left\lceil\log_2L_3\right\rceil,
\qquad
y=\frac{\log x}{8r},
\qquad
\varepsilon=\frac1{13y}.
\end{equation}
Then $z=y$, and
\begin{equation}\label{eq:two-r}
2^{-r}\le L_3^{-1}.
\end{equation}
Using the bound $\vartheta(y)\le2y$ from Section~\ref{sec:preliminaries}, for all sufficiently large $x$,
\[
\log M_{y,r}=r\vartheta(y)\le2ry=\frac14\log x.
\]
Hence
\begin{equation}\label{eq:M-bound}
M_{y,r}\le x^{1/4}.
\end{equation}
For fixed $h$, Mertens' estimate gives
\[
\begin{aligned}
H_h(y)
&=\sum_{5\le p\le y}\frac1p-\sum_{\substack{p\mid h\\5\le p\le y}}\frac1p\\
&=\log\log y+O_h(1).
\end{aligned}
\]
Moreover,
\[
\log y=\log\log x-\log(8r),
\qquad
r=O(\log L_3),
\]
so $\log\log y=L_3+o(1)$. Consequently, after increasing a threshold depending on $h$ if necessary,
\begin{equation}\label{eq:H-lower}
H_h(y)\ge cL_3
\end{equation}
for an absolute constant $c>0$. Therefore
\begin{equation}\label{eq:H-upper}
H_h(y)^{-1/2}\ll L_3^{-1/2},
\end{equation}
with an absolute implied constant once $x\ge x_0(h)$.

The remaining terms in \eqref{eq:master} are smaller. For fixed $H=|h|$,
\[
\frac{H}{x\varepsilon}=\frac{13Hy}{x}=o(L_3^{-1/2}),
\]
and
\[
\frac{1+H/x}{\varepsilon y\log y}
=
\frac{13(1+H/x)}{\log y}=o(L_3^{-1/2}).
\]
Furthermore, by \eqref{eq:two-r},
\[
(1+H/x)2^{-r}=O(L_3^{-1})=o(L_3^{-1/2}),
\]
while \eqref{eq:M-bound} gives
\[
\frac{M_{y,r}}x\le x^{-3/4}=o(L_3^{-1/2}).
\]
Substituting these estimates and \eqref{eq:H-upper} into Proposition~\ref{prop:master} yields
\[
A_{h,\lambda}(x)
\le
C\frac{x}{\sqrt{\log\log\log x}}
\]
simultaneously for all real $\lambda>0$, where $C$ is absolute. The threshold required for \eqref{eq:H-lower}, for the elementary prime estimates, and for $x$ to dominate the fixed quantity $H$ may depend on $h$. This proves Theorem~\ref{thm:uniform-levelset}.

Finally,
\[
0\le\sup_{\lambda>0}\frac{A_{h,\lambda}(x)}x
\le
\frac{C}{\sqrt{\log\log\log x}}\longrightarrow0,
\]
which proves the stated density consequence.
\section{Conditional existence and structural consequences}\label{sec:conditional}

\subsection{The multiplier-two family}

For $t\ge1$, define
\[
\begin{aligned}
f_1(t)&=252t+223, & f_2(t)&=6t+5,\\
f_3(t)&=7t+6, & f_4(t)&=36t+31,
\end{aligned}
\]
and put
\[
n_t=f_1(t)f_2(t),\qquad N_t=n_t+1.
\]
A direct calculation gives
\[
n_t=1512t^2+2598t+1115
\]
and
\begin{equation}\label{eq:Nt}
N_t=6f_3(t)f_4(t).
\end{equation}
Define
\[
\mathcal T
=
\{t\in\N:f_1(t),f_2(t),f_3(t),f_4(t)\text{ are all prime}\},
\]
and
\[
F(x)=\#\{t\in\mathcal T:n_t\le x\}.
\]
We use Schinzel's Hypothesis $H$ in its standard prime-values form
\cite{SchinzelSierpinski1958,SchinzelSierpinskiCorr1959}, and the
Bateman--Horn conjecture in the form of \cite{BatemanHorn1962}.

\begin{theorem}\label{thm:conditional-family}
The following assertions hold.

\textup{(i)} Assume Schinzel's Hypothesis $H$. Then $\mathcal T$ is infinite, and for every $t\in\mathcal T$,
\[
\sigma(N_t)=2\sigma(n_t).
\]
Moreover, the integers $n_t$, $t\in\mathcal T$, are pairwise distinct. Consequently, under Schinzel's Hypothesis $H$, the equation
\[
\sigma(n+1)=2\sigma(n)
\]
has infinitely many positive integer solutions.

\textup{(ii)} Assume the Bateman--Horn conjecture for the four linear forms $f_1,f_2,f_3,f_4$. For each prime $p$, let $\nu(p)$ denote the number of residue classes modulo $p$ on which
\[
f_1(t)f_2(t)f_3(t)f_4(t)
\]
vanishes, and put
\[
\mathfrak S
=
\prod_p\left(1-\frac1p\right)^{-4}\left(1-\frac{\nu(p)}p\right).
\]
Then $\mathfrak S>0$, and
\[
F(x)
\sim
\frac{8\mathfrak S}{3\sqrt{42}}
\frac{\sqrt{x}}{(\log x)^4}.
\]
In particular,
\[
F(x)\gg\frac{\sqrt{x}}{(\log x)^4}.
\]
\end{theorem}
\begin{proof}
We first prove part~(i). Let $t\in\mathcal T$. Since
$f_1(t)-f_2(t)=246t+218>0$, the primes $f_1(t)$ and $f_2(t)$ are
distinct. Similarly, $f_4(t)-f_3(t)=29t+25>0$, so $f_3(t)$ and
$f_4(t)$ are distinct. Since $t\in\mathcal T$, the latter two
numbers are primes exceeding $3$, and hence are coprime to $6$.
Thus multiplicativity of $\sigma$ applies to
\[
n_t=f_1(t)f_2(t),
\qquad
N_t=6f_3(t)f_4(t).
\] Thus
\[
\begin{aligned}
\sigma(n_t)
&=(252t+224)(6t+6)\\
&=28(9t+8)\cdot6(t+1)\\
&=168(t+1)(9t+8),
\end{aligned}
\]
whereas
\[
\begin{aligned}
\sigma(N_t)
&=\sigma(6)(7t+7)(36t+32)\\
&=12\cdot7(t+1)\cdot4(9t+8)\\
&=336(t+1)(9t+8).
\end{aligned}
\]
Consequently,
\begin{equation}\label{eq:multiplier-two}
\sigma(N_t)=2\sigma(n_t).
\end{equation}

We next verify the hypotheses of Schinzel's Hypothesis $H$. Each
$f_i$ is a nonconstant irreducible linear polynomial with positive
leading coefficient, and each is primitive; indeed,
$(252,223)=(6,5)=(7,6)=(36,31)=1$. For every prime $p\ge5$, the
product $f_1f_2f_3f_4$ has at most four zeros modulo $p$, and hence
fewer than $p$. For $p=2$ and $p=3$, a direct calculation shows that
the product has exactly one zero modulo $p$. Thus $f_1f_2f_3f_4$ has no
fixed prime divisor. Schinzel's Hypothesis $H$ therefore implies that $\mathcal T$ is infinite.

Finally,
\[
n_{t+1}-n_t=3024t+4110>0\qquad(t\ge0),
\]
so $t\mapsto n_t$ is strictly increasing. Hence the integers $n_t$ with $t\in\mathcal T$ are pairwise distinct. This proves part~(i).

For part~(ii), write the four forms as $a_it+b_i$. If a prime $p$ divides neither a leading coefficient $a_i$ nor any of the pairwise determinants $a_ib_j-a_jb_i$, then all four forms have well-defined roots modulo $p$, and these roots are pairwise distinct. For the present forms, the six pairwise determinants are
\[
-78,\quad-49,\quad-216,\quad1,\quad6,\quad1.
\]
Together with the prime divisors of the leading coefficients, these
show that only $2,3,7,13$ can differ from the generic four-root case. Direct calculation gives
$\nu(2)=1$, $\nu(3)=1$, $\nu(7)=2$, and $\nu(13)=3$, whereas
$\nu(p)=4$ for $p\notin\{2,3,7,13\}$.
Every local factor in $\mathfrak S$ is therefore positive. Moreover, for $p\notin\{2,3,7,13\}$,
\[
\left(1-\frac4p\right)\left(1-\frac1p\right)^{-4}
=
1-\frac6{p^2}+O(p^{-3}).
\]
Since $\sum_p\frac1{p^2}<\infty$, the product over the nonexceptional primes converges absolutely to a positive limit. The finitely many exceptional local factors are also positive. Hence $\mathfrak S>0$.

Let
\[
\Pi(T)=\#\{1\le t\le T:t\in\mathcal T\}.
\]
The Bateman--Horn conjecture gives
\[
\Pi(T)
\sim
\mathfrak S\int_2^T\frac{du}{(\log u)^4}
\sim
\mathfrak S\frac{T}{(\log T)^4}.
\]
Since
\[
n_t=1512t^2+2598t+1115,
\]
the largest integer $t$ for which $n_t\le x$ is, for all sufficiently large $x$,
\[
T_x
=
\left\lfloor
\frac{\sqrt{6048x+6084}-2598}{3024}
\right\rfloor.
\]
Thus
\[
T_x\sim\sqrt{\frac{x}{1512}},
\qquad
\log T_x=\frac12\log x+O(1).
\]
By the strict monotonicity proved above,
\[
F(x)=\Pi(T_x).
\]
It follows that
\[
\begin{aligned}
F(x)
&\sim\mathfrak S\frac{T_x}{(\log T_x)^4}\\
&\sim\frac{16\mathfrak S}{\sqrt{1512}}\frac{\sqrt{x}}{(\log x)^4}\\
&=\frac{8\mathfrak S}{3\sqrt{42}}\frac{\sqrt{x}}{(\log x)^4}.
\end{aligned}
\]
This proves part~(ii).
\end{proof}

\subsection{Structural properties of the multiplier-two family}

Let $t\in\mathcal T$, and put
\[
p=7t+6,\qquad q=36t+31,\qquad m=pq.
\]
Then $N_t=6m$, where $p$ and $q$ are distinct primes exceeding $3$. In particular, $(m,6)=1$.

\begin{proposition}\label{prop:structural-family}
For every $t\in\mathcal T$, the integer $N_t$ is Zumkeller, semiperfect, and abundant; moreover, it is neither weird nor practical. More precisely, the two sets
\[
\{6d:d\mid m\}
\]
and
\[
\{d:d\mid m\}\cup\{2d:d\mid m\}\cup\{3d:d\mid m\}
\]
form a partition of $D(N_t)$ into classes of equal sum, and
\[
\sigma(N_t)-2N_t=516t+456.
\]
\end{proposition}
\begin{proof}
Since $(6,m)=1$, every positive divisor of $N_t=6m$ has a unique representation
\[
ed,\qquad e\mid6,\quad d\mid m.
\]
Hence $D(N_t)$ is the disjoint union
\[
\{d:d\mid m\},\qquad
\{2d:d\mid m\},\qquad
\{3d:d\mid m\},\qquad
\{6d:d\mid m\}.
\]
The sum of the divisors in the last class is $6\sigma(m)$, while the sum over the other three classes is
\[
(1+2+3)\sigma(m)=6\sigma(m).
\]
Thus the displayed two sets form a Zumkeller partition of $D(N_t)$.

Furthermore, $m,2m,3m$ are distinct proper divisors of $N_t$, and
$m+2m+3m=N_t$. Hence $N_t$ is semiperfect.

Since $p$ and $q$ are distinct primes coprime to $6$,
\[
\sigma(N_t)=12(p+1)(q+1).
\]
Therefore
\[
\begin{aligned}
\sigma(N_t)-2N_t
&=12\bigl((p+1)(q+1)-pq\bigr)\\
&=12(p+q+1)\\
&=12(43t+38)\\
&=516t+456>0.
\end{aligned}
\]
Thus $N_t$ is abundant. Since it is semiperfect, it is not weird.

It remains to rule out practicality. We have
\[
N_t=2\cdot3\cdot p\cdot q,
\qquad
2<3<p<q.
\]
By the Stewart--Sierpi\'nski characterization in Theorem~\ref{thm:practical}, practicality would require
\[
p\le1+\sigma(2\cdot3)=13.
\]
Since $p=7t+6$, this forces $t=1$. Then $p=13$ and $q=67$, and
$67\le1+\sigma(2\cdot3\cdot13)=169$. However,
$f_1(1)=475$ is composite, so $1\notin\mathcal T$. Therefore $N_t$ is not practical for any $t\in\mathcal T$.
\end{proof}

No analogous divisor-partition conclusion is asserted for arbitrary solutions of
\[
\sigma(n+h)=k\sigma(n).
\]

\begin{corollary}\label{cor:structural-BH}
Assume the Bateman--Horn conjecture for $f_1,f_2,f_3,f_4$. Then
\[
\#\left\{t\in\mathcal T:
\begin{array}{l}
n_t\le x,\\
N_t\text{ is Zumkeller, semiperfect, abundant,}\\
N_t\text{ is non-weird and non-practical}
\end{array}
\right\}
\sim
\frac{8\mathfrak S}{3\sqrt{42}}\frac{\sqrt{x}}{(\log x)^4}.
\]
\end{corollary}
\begin{proof}
Proposition~\ref{prop:structural-family} applies to every $t\in\mathcal T$. Hence the counting function in the corollary is precisely $F(x)$, and the assertion follows from Theorem~\ref{thm:conditional-family}\textup{(ii)}.
\end{proof}

\subsection{Conjectures and an open problem}

The preceding results concern the explicit multiplier-two family above.
They lead naturally to the following general infinitude conjecture.

\begin{conjecture}\label{conj:generalized-ES}
For every pair of positive integers $h$ and $k$,
\[
\#\{n\ge1:\sigma(n+h)=k\sigma(n)\}=\infty.
\]
\end{conjecture}
The case $(h,k)=(1,1)$ is the classical Erd\H{o}s--Sierpi\'nski conjecture.

Corollary~\ref{cor:integer-multiplier} shows, on the other hand, that for every fixed $h\ge1$ these conjecturally infinite solution sets are nevertheless
quantitatively sparse:
\[
\sup_{k\in\N}A_{h,k}(x)\ll_h\frac{x}{(\log x)^2}.
\]
The exceptional set introduced in Section~\ref{sec:divisibility} suggests a second, more specific conjecture.

\begin{conjecture}\label{conj:V1}
One has $\mathcal V_1=\varnothing$. Equivalently,
\[
\frac{I(v)}{I(v+1)}\notin\N\qquad(v\ge2).
\]
\end{conjecture}
Proposition~\ref{prop:integral-abundancy} gives
\[
v\in\mathcal V_1\Longrightarrow v\mid\sigma(v),
\]
so every hypothetical member of $\mathcal V_1$ would be multiperfect. Moreover, Conjecture~\ref{conj:V1} together with Corollary~\ref{cor:dichotomy}\textup{(i)} would imply
\[
B_1(x)
\le
x\exp\!\left(
-\left(\frac1{\sqrt6}+o(1)\right)\sqrt{\log x\,\log\log x}
\right).
\]
The conjecture is restricted to the consecutive shift. Indeed, as noted in Remark~\ref{rem:V22},
\[
6\in\mathcal V_{22},\qquad I(6)=I(28)=2.
\]

\begin{openproblem}\label{prob:characterize-Vh}
Characterize the shifts $h\ge1$ for which $\mathcal V_h\ne\varnothing$.
\end{openproblem}

\printbibliography

@article{CanfieldErdosPomerance1983,
  author  = {Canfield, E. R. and Erd{\H{o}}s, Paul and Pomerance, Carl},
  title   = {On a problem of Oppenheim concerning ``Factorisatio Numerorum''},
  journal = {Journal of Number Theory},
  volume  = {17},
  number  = {1},
  pages   = {1--28},
  year    = {1983},
  doi     = {10.1016/0022-314X(83)90002-1}
}

@article{HildebrandTenenbaum1993,
  author  = {Hildebrand, Adolf and Tenenbaum, G{\'e}rald},
  title   = {Integers without large prime factors},
  journal = {Journal de Th{\'e}orie des Nombres de Bordeaux},
  volume  = {5},
  number  = {2},
  pages   = {411--484},
  year    = {1993},
  doi     = {10.5802/jtnb.101}
}

@book{HalberstamRichert1974,
  author    = {Halberstam, Heini and Richert, Hans-Egon},
  title     = {Sieve Methods},
  series    = {London Mathematical Society Monographs},
  volume    = {4},
  publisher = {Academic Press},
  year      = {1974}
}

@article{LucaPomerance2007,
  author  = {Luca, Florian and Pomerance, Carl},
  title   = {On the average number of divisors of the {Euler} function},
  journal = {Publicationes Mathematicae Debrecen},
  volume  = {70},
  number  = {1--2},
  pages   = {125--148},
  year    = {2007},
  doi     = {10.5486/PMD.2007.3476}
}

@article{LucaPomeranceCorr2016,
  author  = {Luca, Florian and Pomerance, Carl},
  title   = {Corrigendum: ``On the average number of divisors of the {Euler} function''},
  journal = {Publicationes Mathematicae Debrecen},
  volume  = {89},
  number  = {1--2},
  pages   = {257--260},
  year    = {2016},
  doi     = {10.5486/PMD.2016.7645}
}

@article{Kesten1969,
  author  = {Kesten, Harry},
  title   = {A sharper form of the {Doeblin--L{\'e}vy--Kolmogorov--Rogozin} inequality for concentration functions},
  journal = {Mathematica Scandinavica},
  volume  = {25},
  pages   = {133--144},
  year    = {1969},
  doi     = {10.7146/math.scand.a-10950}
}

@article{MahantaSaikiaYaqubi2020,
  author  = {Mahanta, Pankaj Jyoti and Saikia, Manjil P. and Yaqubi, Daniel},
  title   = {Some properties of {Zumkeller} numbers and $k$-layered numbers},
  journal = {Journal of Number Theory},
  volume  = {217},
  pages   = {218--236},
  year    = {2020},
  doi     = {10.1016/j.jnt.2020.05.003}
}

@online{Jokar2022,
  author      = {Jokar, Farshid},
  title       = {On $k$-layered numbers},
  year        = {2022},
  eprint      = {2207.09053},
  eprinttype  = {arXiv},
  eprintclass = {math.NT},
  note        = {Preprint}
}

@article{Stewart1954,
  author  = {Stewart, B. M.},
  title   = {Sums of distinct divisors},
  journal = {American Journal of Mathematics},
  volume  = {76},
  number  = {4},
  pages   = {779--785},
  year    = {1954},
  doi     = {10.2307/2372651}
}

@article{Sierpinski1955,
  author  = {Sierpi{\'n}ski, Wac{\l}aw},
  title   = {Sur une propri{\'e}t{\'e} des nombres naturels},
  journal = {Annali di Matematica Pura ed Applicata},
  volume  = {39},
  number  = {1},
  pages   = {69--74},
  year    = {1955},
  doi     = {10.1007/BF02410762}
}

@article{GuyShanks1974,
  author  = {Guy, Richard K. and Shanks, Daniel},
  title   = {A constructed solution of $\sigma(n)=\sigma(n+1)$},
  journal = {The Fibonacci Quarterly},
  volume  = {12},
  number  = {3},
  pages   = {299},
  year    = {1974}
}

@article{Makowski1960,
  author  = {M{\k{a}}kowski, Andrzej},
  title   = {On some equations involving functions $\phi(n)$ and $\sigma(n)$},
  journal = {The American Mathematical Monthly},
  volume  = {67},
  number  = {7},
  pages   = {668--670},
  year    = {1960},
  doi     = {10.2307/2310107}
}

@article{MakowskiCorr1961,
  author  = {M{\k{a}}kowski, Andrzej},
  title   = {Correction},
  journal = {The American Mathematical Monthly},
  volume  = {68},
  number  = {7},
  pages   = {650},
  year    = {1961}
}

@article{HunsuckerNebbStearns1973,
  author  = {Hunsucker, J. L. and Nebb, J. and Stearns, Jr., R. E.},
  title   = {Computational results concerning some equations involving $\sigma(n)$},
  journal = {The Mathematics Student},
  volume  = {41},
  pages   = {285--289},
  year    = {1973}
}

@article{MientkaVogt1970,
  author  = {Mientka, W. E. and Vogt, R. L.},
  title   = {Computational results relating to problems concerning $\sigma(n)$},
  journal = {Matemati\v{c}ki Vesnik},
  volume  = {7(22)},
  number  = {51},
  pages   = {35--36},
  year    = {1970}
}

@article{ErdosPomeranceSarkozy1987,
  author  = {Erd{\H{o}}s, Paul and Pomerance, Carl and S{\'a}rk{\H{o}}zy, Andr{\'a}s},
  title   = {On locally repeated values of certain arithmetic functions, {II}},
  journal = {Acta Mathematica Hungarica},
  volume  = {49},
  number  = {1--2},
  pages   = {251--259},
  year    = {1987}
}

@article{Yamada2017,
  author  = {Yamada, Tomohiro},
  title   = {On equations $\sigma(n)=\sigma(n+k)$ and $\varphi(n)=\varphi(n+k)$},
  journal = {Journal of Combinatorics and Number Theory},
  volume  = {9},
  number  = {1},
  pages   = {15--21},
  year    = {2017}
}

@article{Weingartner2011,
  author  = {Weingartner, Andreas},
  title   = {On the solutions of $\sigma(n)=\sigma(n+k)$},
  journal = {Journal of Integer Sequences},
  volume  = {14},
  number  = {5},
  pages   = {Article 11.5.5, 7 pp.},
  year    = {2011}
}

@article{PollackPomerance2016,
  author  = {Pollack, Paul and Pomerance, Carl},
  title   = {Some problems of Erd{\H{o}}s on the sum-of-divisors function},
  journal = {Transactions of the American Mathematical Society, Series B},
  volume  = {3},
  pages   = {1--26},
  year    = {2016},
  doi     = {10.1090/btran/10}
}

@article{Ford2022,
  author  = {Ford, Kevin},
  title   = {Solutions of $\phi(n)=\phi(n+k)$ and $\sigma(n)=\sigma(n+k)$},
  journal = {International Mathematics Research Notices},
  volume  = {2022},
  number  = {5},
  pages   = {3561--3570},
  year    = {2022},
  doi     = {10.1093/imrn/rnaa218}
}

@article{Mangerel2024,
  author  = {Mangerel, Alexander P.},
  title   = {On equal consecutive values of multiplicative functions},
  journal = {Discrete Analysis},
  volume  = {2024},
  pages   = {Paper No. 12, 20 pp.},
  year    = {2024},
  doi     = {10.19086/da.125450}
}

@article{SchinzelSierpinski1958,
  author  = {Schinzel, Andrzej and Sierpi{\'n}ski, Wac{\l}aw},
  title   = {Sur certaines hypoth{\`e}ses concernant les nombres premiers},
  journal = {Acta Arithmetica},
  volume  = {4},
  number  = {3},
  pages   = {185--208},
  year    = {1958},
  doi     = {10.4064/aa-4-3-185-208}
}

@article{SchinzelSierpinskiCorr1959,
  author  = {Schinzel, Andrzej and Sierpi{\'n}ski, Wac{\l}aw},
  title   = {Remarque concernant le travail de A. Schinzel et W. Sierpi{\'n}ski: ``Sur certaines hypoth{\`e}ses concernant les nombres premiers''},
  journal = {Acta Arithmetica},
  volume  = {5},
  number  = {2},
  pages   = {259},
  year    = {1959},
  doi     = {10.4064/aa-5-2-259}
}

@article{BatemanHorn1962,
  author  = {Bateman, Paul T. and Horn, Roger A.},
  title   = {A heuristic asymptotic formula concerning the distribution of prime numbers},
  journal = {Mathematics of Computation},
  volume  = {16},
  number  = {79},
  pages   = {363--367},
  year    = {1962},
  doi     = {10.1090/S0025-5718-1962-0148632-7}
}

@article{TorabiFatehizadeh2021,
  author  = {Torabi, Hamid and Fatehizadeh, Amirali},
  title   = {A generalization of the Erdos-Serpinski conjecture},
  journal = {Journal of New Researches in Mathematics},
  volume  = {6},
  number  = {28},
  pages   = {75--82},
  year    = {2021},
  note    = {[in Persian]}
}

@online{Fatehizadeh2026,
  author      = {Fatehizadeh, Amirali},
  title       = {A generalization of the {Erd{\H{o}}s-Sierpi{\'n}ski} conjecture},
  year        = {2026},
  eprint      = {2605.21524},
  eprinttype  = {arXiv},
  eprintclass = {math.NT},
  note        = {Preprint}
}

@online{Benito2007,
  author      = {Benito, Lourdes},
  title       = {Solutions of the problem of {Erd{\"o}s-Sierpi{\'n}ski}: $\sigma(n)=\sigma(n+1)$},
  year        = {2007},
  eprint      = {0707.2190},
  eprinttype  = {arXiv},
  eprintclass = {math.NT},
  note        = {Preprint}
}

@article{BaylessKinlaw2015,
  author  = {Bayless, Jonathan and Kinlaw, Paul},
  title   = {On repeated values of $\sigma$ and multiperfect numbers},
  journal = {Journal of Combinatorics and Number Theory},
  volume  = {7},
  number  = {3},
  pages   = {177--189},
  year    = {2015}
}

@online{OEIS2026,
  author = {{OEIS Foundation Inc.}},
  title  = {The On-Line Encyclopedia of Integer Sequences},
  year   = {2026},
  url    = {https://oeis.org}
}
\nocite{Fatehizadeh2026}
\end{document}